\documentclass{amsart}
\usepackage{amssymb}
\usepackage[all]{xy}
\usepackage{enumitem}

\def\Z{{\mathbb Z}}
\def\Q{{\mathbb Q}}
\def\R{{\mathbb R}}
\def\C{{\mathbb C}}

\def\F{{\mathbb F}}

\def\calI{{\mathcal I}}

\def\br{{\mathbf r}}
\def\bx{{\mathbf x}}
\def\bh{{\mathbf h}}

\def\Ghat{\widehat{G}}
\def\Xhat{{\widehat{X}}}

\def\trace{\mathrm{trace}}

\newcommand\Aut{\operatorname{Aut}}

\newcommand\Err{\operatorname{Err}}

\newtheorem{theorem}{Theorem}[section]
\newtheorem{lemma}[theorem]{Lemma}
\newtheorem{proposition}[theorem]{Proposition}
\newtheorem{corollary}[theorem]{Corollary}
\newtheorem{bigtheorem}{Theorem}
\newtheorem{bigcorollary}[bigtheorem]{Corollary}

\theoremstyle{definition}
\newtheorem{definition}[theorem]{Definition}
\newtheorem{example}[theorem]{Example}

\theoremstyle{remark}
\newtheorem{remark}[theorem]{Remark}

\begin{document}

\title[Integration over finite groups and difference sets]
{Approximation of integration
over finite groups, difference sets and association schemes}

\author{Hiroki Kajiura}
\address{JSPS Research Fellow.
Department of Mathematics\\ Graduate School of Science\\
Hiroshima University, 739-8526 Japan}
\email{hikajiura@hiroshima-u.ac.jp}

\author{Makoto Matsumoto}
\address{Mathematics Program \\ 
Graduate School of Advanced Science and Engineering\\
Hiroshima University, 739-8526 Japan}
\email{m-mat@math.sci.hiroshima-u.ac.jp}

\author{Takayuki Okuda}
\address{Mathematics Program \\ 
Graduate School of Advanced Science and Engineering\\
Hiroshima University, 739-8526 Japan}
\email{okudatak@hiroshima-u.ac.jp}

\thanks{
The first author is partially supported
by JSPS Grant-in-Aid for JSPS Fellows Grant Number JP19J21207,
the second author by JSPS
Grants-in-Aid for Scientific Research
JP26310211 and JP18K03213, 
and the third author
by JP16K17594, JP16K05132, JP16K13749,
and JP26287012.
}

\date{\today}

\begin{abstract}
Let $G$ be a finite group and $f:G \to \C$ be a function.
For a non-empty finite subset $Y\subset G$, 
let $I_Y(f)$ denote the average of $f$ over $Y$. Then,
$I_G(f)$ is the average of $f$ over $G$. 
Using the decomposition of $f$ into irreducible components 
of $\C^G$ as a representation of $G\times G$,
we define
non-negative real numbers $V(f)$ and $D(Y)$, each depending
only on $f$, $Y$, respectively, such that
an inequality of the form $|I_G(f)-I_Y(f)|\leq V(f)\cdot D(Y)$
holds. We give a lower bound of $D(Y)$ depending only
on $\#Y$ and $\#G$. We show that the lower bound is achieved
if and only if 
$\#\{(x,y)\in Y^2 \mid x^{-1}y \in [a]\}/\#[a]$
is independent of the choice of the conjugacy class $[a]\subset G$
for $a \neq 1$. We call such a $Y\subset G$ as a pre-difference set
in $G$, since the  
condition is satisfied if $Y$ is a difference set. If $G$ is abelian,
the condition is equivalent to that $Y$ is a difference set.
We found a non-trivial pre-difference set in the dihedral group 
of order 16, where no non-trivial difference set exists.
The pre-difference sets in non-abelian groups of order 16
are classified.
A generalization to commutative association schemes is 
also given.
\end{abstract}

\maketitle

\section{Introduction and main results}\label{sec:introduction}
Let $X$ be a non-empty finite set, and $Y$ be a non-empty subset of $X$.
We denote by $\C^X$ the space of functions from $X$
to $\C$.
For $f \in \C^X$, its integration $I(f)$ over $X$ is defined
as $\frac{1}{\#X}\sum_{x\in X} f(x)$. We use the term ``integration''
because of a similarity to the theory of quasi-Monte Carlo
integration, stated later.
Similarly, the integration $I_Y(f)$ over $Y$ is defined
as $\frac{1}{\#Y}\sum_{x\in Y} f(x)$. 
We would like to find a finite subset so that the 
absolute integration
error
$|I(f)-I_Y(f)|$ is small, for $f$ being in
some function space $F \subset \C^X$.
This is an analogue of quasi-Monte Carlo (QMC) methods
in approximating the integration, where
$X$ is a hyper cube $[0,1]^s$ and the integration 
of $f:X \to \R$ is with respect to the Lebesgue measure. 
A large amount of studies exist, see for example
\cite{DICK-KUO-SLOAN} \cite{DICK-PILL-BOOK} \cite{niederreiter:book}.
Some recent researches would be found 
in the conference books from international conferences titled 
``Monte Carlo and quasi Monte Carlo methods''
held every other year. For the readers' convenience,
we briefly explain a typical QMC and its relation with
group characters in Appendix~\ref{sec:appendix}.

In this manuscript, we consider the case where $X$ is
a finite group $G$, and in Section~\ref{sec:as}
a generalization where $X$ has 
a structure of commutative association scheme.

Here, $\C^X$ is equipped with a standard inner product
and a norm: 
$$
\left< f, g \right>:=\sum_{x\in X} f(x)\overline{g(x)}, \quad
||f|| := \sqrt{\left<f , f \right>}.
$$

When $G$ is a finite group, 
$\C^G$ is a left $G$-module by defining the action $g \in G$ on $f(-)\in \C^G$
by $g(f(-))=f(g^{-1}(-))$.
Then $\C^G$ has an orthogonal decomposition 
$\C^G=\oplus_{\rho \in \Ghat} V_\rho$,
where $\Ghat$ is the set of isomorphism classes of 
irreducible characters and $V_\rho$ is the 
submodule isomorphic to the direct sum of the
$\dim \rho$ copies of the representation
$\rho$. Hence, any $f\in \C^G$
is decomposed as $f=\oplus_{\rho \in \Ghat}f_\rho$,
and $f_\rho$ is called the $\rho$-component of $f$

We define non-negative real numbers $\partial_\rho(Y)$
(which is $||\calI_Y^\rho||$ defined in Proposition~\ref{prop:rho-comp2})
such that $|I(f_\rho)-I_Y(f_\rho)|\leq ||f_\rho||\partial_\rho(Y)$ holds,
which is sharp since there is an $f$ with equality 
holds for all $\rho$ (see Remark~\ref{rem:equality} below).

\begin{bigtheorem}\label{Th:main1}
Let $G$ be a finite group, $Y$ a nonempty subset of $G$,
and $f:G \to \C$ a function. Let $I(f)$ be the
average of $f$ over $G$, $I_Y(f)$ the average of $f$
over $Y$. Let $f_\rho$ be the $\rho$-component of 
$f$, where $\rho$ is an irreducible representation of $G$, 
namely, $\rho \in \Ghat$.
Define 
$$
 \partial_\rho(Y):=
 \sqrt{\frac{\dim \rho}{\#Y^2\# G} \sum_{x,y\in Y}\chi_\rho(x^{-1}y)},
$$
where $\chi_\rho$ is the character of $\rho$.
Then we have
$|I(f_\rho)-I_Y(f_\rho)| \leq ||f_\rho||\partial_\rho(Y)$.
For the trivial representation $\rho=1_G$, we have
$|I(f_{1_G})-I_Y(f_{1_G})|=0$.
\end{bigtheorem}
A proof is given in Section~\ref{sec:proof-th1}.
\begin{bigcorollary}\label{cor:weight}
We have 
$$
|I(f)-I_Y(f)| \leq 
\sum_{\rho \in \Ghat \setminus\{1_G\}}||f_\rho||\partial_\rho(Y)
\leq  V(f)D(Y),
$$
where
$V(f):=\sum_{\rho\in \Ghat \setminus\{1_G\}}
\left(||f_\rho||\dim \rho\right)$
and  
$D(Y):=\max_{\rho \in \Ghat\setminus \{1_G\}} 
\frac{\partial_\rho(Y)}{\dim \rho}$.
\end{bigcorollary}
\begin{proof}
We have $I(f)=\sum_{\rho \in \Ghat} I(f_\rho)$.
The above theorem implies the second inequality
in
\begin{eqnarray*}
 |I(f)-I_Y(f)|
&\leq& \sum_{\rho}|I(f_\rho)-I_Y(f_\rho)|
 \leq \sum_{\rho \neq 1_G}||f_\rho||\partial_\rho(Y) \\
&=& \sum_{\rho \neq 1_G}||f_\rho|| \dim\rho
   \cdot \frac{\partial_\rho(Y)}{\dim\rho}
 \leq \left(\sum_{\rho \neq 1_G}||f_\rho||\dim\rho\right)
 D(Y).
\end{eqnarray*}
\end{proof}

Note that 
$V(f)$
depends only on $f$,
and $D(Y)$ depends only on $Y$. 
Corollary~\ref{cor:weight} implies
\begin{equation*}
 |I(f)-I_Y(f)| \leq V(f)D(Y).
\end{equation*}
Such type of error bounds on QMC-integration appears
in many researches; a famous example is
Koksma-Hlawka inequality \cite{niederreiter:book}.

It is easy to show that $D(G)=0$,
so this bound is tight in this sense.
Because of this inequality, we are interested in finding a $Y$
with fixed cardinality
which minimizes $D(Y)$. 
\begin{proposition}\label{prop:main2}
Let $D(Y)$ be the 
value defined in Corollary~\ref{cor:weight}.
Under the assumptions in Theorem~\ref{Th:main1},
a lower bound 
$$
D(Y) \geq 
\sqrt{\frac{1/\#Y-1/\#G}{\#G-1}}
$$
holds.
The equality holds 
if and only if 
$\frac{\partial_\rho(Y)}{\dim \rho} 
= \sqrt{\frac{1/\#Y-1/\#G}{\#G-1}}$
holds for any $\rho \neq 1_G$.
\end{proposition}
A proof is given in Section~\ref{sec:main3}.
Thus,
$D(Y)$ is bounded below by 
a formula depending only on $\#G$ and $\#Y$.
We are interested in the case where the equality holds.
If $G$ is an abelian group, then the equality 
holds
if and only if $Y$ is a difference set\cite{difference-set}, as shown
in Theorem~\ref{Th:main2}.
\begin{definition}
Let $G$ be a finite group, and $Y$ its subset.
Define
$\lambda_a:=\#\{(x,y) \in Y\times Y \mid x^{-1}y = a\}$,
where $a\in G$.
A non-empty proper subset $Y$ of $G$ is
a $(v,k,\lambda)$-difference set if
$v=\#G, k=\#Y$, and $\lambda_a=\lambda$
for any $a \in G$ except $a=1$.
Note that $\lambda_1=\#Y$. 
A difference set $Y$ is said to be trivial
if $\#(Y)=1, v-1$.
\end{definition}

We define the notion of pre-difference set.
We could not find this notion in the literatures. 
\begin{definition}\label{def:pre}(pre-difference set)

Let $G$ be a finite group, and $Y$ its non-empty
subset.
Define $\lambda_{[a]}:=\#\{(x,y) \in Y\times Y \mid x^{-1}y \in [a]\}$,
for the conjugacy class $[a]\subset G$ of $a\in G$.
Then $Y$ is a $(v,k,\lambda)$ pre-difference set if
$v=\#G, k=\#Y$, and $\lambda_{[a]}=\lambda\#[a]$
for any $a \in G$ except $a=1$.
Note that $\lambda_{[1]}=\#Y$. 
A pre-difference set $Y$ is said to be trivial
if $\#(Y)=1, v-1, v$.
\end{definition}

If $Y$ is a difference set, then it is a pre-difference
set. The converse is not true, because there is a non-trivial 
pre-difference set in the dihedral group of order 16,
see Section~\ref{sec:pre}. The parameter $\lambda$ for
pre-difference set appears to be a rational number,
but we shall show that it is an integer in Section~\ref{sec:pre}.

\begin{remark}
The set $Y=G$ is not considered as a difference set,
but is considered as a pre-difference set. 
This is because in the context of the integration,
$Y=G$ is the best choice and we don't want to exclude.
\end{remark}

\begin{bigtheorem}\label{Th:main3}
Let $G$ be a finite group and $Y$ its non-empty subset.
Then, the following conditions
are equivalent.
\begin{enumerate}
\item $Y$ is a pre-difference set, i.e.,
$\lambda_{[a]}/\#[a]$ is independent of the choice 
of $a \neq 1$.
\item
$\lambda_{[a]}/\#[a]=(\#Y^2-\#Y)/(\#G-1)$
holds for any $a \neq 1$, and $\lambda_{[1]}=\#Y$ holds.
\item
$
 \partial_\rho(Y)/(\dim \rho)=
 \sqrt{\frac{1/\#Y-1/\#G}{\#G-1}}
$
holds for any $\rho \neq 1_G$, and
$\partial_{1_G}(Y)=\sqrt{\frac{1}{\#G}}$ holds.
\item 
$\partial_\rho(Y)/(\dim \rho)$ is independent of 
the choice of $\rho \in \Ghat\setminus\{1_G\}$.
\item 
The equality holds in the inequality
$$
D(Y)
\geq 
\sqrt{\frac{1/\#Y-1/\#G}{\#G-1}}
$$
in Proposition~\ref{prop:main2}.
\end{enumerate}
\end{bigtheorem}
A proof is given in Section~\ref{sec:main3}.
Thus, when the cardinality of $Y$ is fixed,
$D(Y)$ attains the lower bound above if and only if 
$Y$ is a pre-difference set. It seems 
interesting that the optimal choice
for the Quasi-Monte Carlo in the above setting
is equivalent
to a natural generalization of the notion
of classical difference set.

%
%
%

If $G$ is abelian, then $[a]=\{a\}$,
$\lambda_{[a]}=\lambda_a$, and $\dim \rho=1$. 
This shows the following proposition.
\begin{proposition}
Let $G$ be an abelian group. 
A subset $Y\subset G$ is a difference set
if and only if $Y\neq G$ and $Y$ is a pre-difference set.
\end{proposition}
Thus, the above theorem implies
the following theorem.
\begin{bigtheorem}\label{Th:main2}
Let $G$ be a finite abelian group, and $Y$ its 
non-empty subset.
Then the following conditions are equivalent.
\begin{enumerate}
\item $Y$ 
is a difference set or $Y=G$. 
\item $\partial_\rho(Y)$ is independent of 
the choice of $\rho \in \Ghat\setminus\{1_G\}$.
\item 
The equality holds in the inequality
$$
D(Y)
\geq 
\sqrt{\frac{1/\#Y-1/\#G}{\#G-1}}
$$
in Proposition~\ref{prop:main2}.
\end{enumerate}
\end{bigtheorem}

In Section~{\ref{sec:as}}, we prove a theorem generalizing
Theorem~\ref{Th:main3} in the context of commutative association
schemes.
\section{Preliminary for the proofs and a proof of Theorem~\ref{Th:main1}}
\label{sec:proof-th1}
\begin{definition}\label{def:set}
For a finite set $X$, $\C^X$ denotes the set of functions
from $X$ to $\C$, and $\C[X]$ denotes
the linear vector space with basis $X$.
We identify $\C^X =\C[X]$ by $f \mapsto \sum_{x \in X} f(x)x$.
A standard Hermitian inner product and the norm are
defined as in Section~\ref{sec:introduction}.
\end{definition}

\begin{definition}\label{def:I}
For $x\in X$, an element $x \in \C[X]$ corresponds to
the delta function $\delta_x \in \C^X$, defined by
$\delta_x(y)=0$ if $x\neq y$ and $\delta_x(y)=1$ if $x=y$,
for $y \in X$. For a subset $Y$, we define
$\delta_Y=\sum_{y\in Y}\delta_y \in \C^X$.
Note that $\left<\delta_Y, \delta_Z\right>=\#(Y\cap Z)$.
In Section~\ref{sec:introduction}, we defined an operator
taking the average over $Y$:
$I_Y: \C^G \to \C$.
We define $\calI_Y:=\frac{1}{\#Y}\delta_Y \in \C^X$
representing the operator $I_Y$, i.e.,
$$
I_Y(f)=\left<f, \calI_Y\right>
$$
holds. By definition, we have
$$
\left<\calI_Y,\calI_Z\right>=\frac{1}{\#Y\#Z}\#(Y\cap Z).
$$
If $X$ is a group $G$, 
then for $\rho \in \widehat{G}$ (see below), we denote the $\rho$-component
of $\delta_y, \delta_Y, \calI_Y$ by
$\delta_y^\rho, \delta_Y^\rho, \calI_Y^\rho$, 
respectively.
\end{definition}


\begin{definition}
In the above definition, suppose that $G=X$ is a group.
Then, $\C[G]$ is a group ring. By left multiplication,
$\C[G]$ is a left $G$-module.
Let $\Ghat$ denote the set of isomorphism classes of
irreducible representations of $G$.
The character of $\rho$ is denoted by $\chi_\rho$.
Normalized inner product is defined for $\C^G$ (and hence
for $\C[G]$) by 
$$
(f,g)_G:=\frac{1}{\#G}\sum_{x\in G}f(x)\overline{g(x)}.
$$
(This inner product is used only in Proposition~\ref{prop:group}
for stating orthonomality of the characters.)
Let $C(G)$ be the set of conjugacy class, and for
$a\in G$, $[a]\in C(G)$ denotes the class that contains $a$.
The linear subspace $\C^{C(G)}$ of $\C^G=\C[G]$ 
is the center of $\C[G]$.

\end{definition}
The following are well-known,
see for example Serre \cite[I, \S2]{serre}.
\begin{proposition}\label{prop:group}
\begin{enumerate} 
\item $\{\chi_\rho \mid \rho \in \Ghat\}$ is an orthonormal basis
of the space of class functions $\C^{C(G)}$ with
respect to the normalized inner product.
\item As a left $G$-module,
$\C[G]=\oplus_{\rho\in \Ghat} H_\rho$, where
$H_\rho$ is the isotypical component to the irreducible
representation $\rho$. The decomposition is orthogonal.
Thus, $f\in \C[G]$ is decomposed as $f=\oplus_\rho f_\rho$,
and $f_\rho \in H_\rho$ is called the $\rho$-component of $f$.
The multiplicity of $\rho$
in $H_\rho$ is $\dim \rho$. Hence,
$\sum_{\rho \in \Ghat} (\dim \rho)^2=\#G$.
\item
Let 
$$
E_\rho:=\frac{\dim \rho}{\#G}\sum_{g\in G} \overline{\chi_\rho(g)}g \in \C[G].
$$
Then, the left multiplication of $E_\rho$ to $\C[G]$
is the projection $\C[G] \to H_\rho$.
Thus, for $f \in \C^X$, its $\rho$-component $f_\rho$
is $E_\rho f$
(through the identification $\C^G=\C[G]$ as left $G$-modules given 
in Definition~\ref{def:set}).
\end{enumerate}
\end{proposition}
The above proposition implies the following:
\begin{proposition}\label{prop:rho-comp}
For $x \in G$, the $\rho$-component $\delta_x^\rho=E_\rho \delta_x$
of $\delta_x$
 is given 
in $\C[G]$ by
$$
\delta_x^\rho=
\frac{\dim \rho}{\#G}\sum_{g\in G} \overline{\chi_\rho(g)}gx
=
\frac{\dim \rho}{\#G}\sum_{g\in G} \overline{\chi_\rho(gx^{-1})}g.
$$ 
Through the identification $\C[G]=\C^G$, 
$$
\delta_x^\rho(g)
=\frac{\dim \rho}{\#G}\overline{\chi_\rho(gx^{-1})}
=\frac{\dim \rho}{\#G}\chi_\rho(xg^{-1})
=\frac{\dim \rho}{\#G}\chi_\rho(g^{-1}x).
$$
\end{proposition}

%

\begin{proposition}\label{prop:rho-comp2}
Let $Y$ be a subset of $G$. Recall Definition~\ref{def:I}
for various elements in $\C^G$ related with $Y$.
For $f \in \C^G$, 
$$
I_Y(f)=\left<f, \calI_Y \right>
=\sum_{\rho\in \Ghat}
\left<f, \calI_Y^\rho \right>
=\sum_{\rho\in \Ghat}
\left<f_\rho, \calI_Y^\rho \right>.
$$
\end{proposition}
\begin{proof}
The first equality is shown in Definition~\ref{def:I}.
The second follows from $\calI_Y=\sum_{\rho}\calI_Y^\rho$.
The third follows from the orthogonality of $H_\rho$.
\end{proof}

\paragraph{Proof of Theorem~\ref{Th:main1}.}
Let $1_G$ denote the trivial representation.
Note that for any $f \in \C^G$, $I(f)=I_G(f)$, and $f_{1_G}=E_{1_G}(f)=I(f)$
(where $I(f)$ denotes the constant function with 
value $I(f)$ in $\C^G$).

Since $G$ acts on $\calI_G$ trivially, $\calI_G^\rho=0$ 
for $\rho\neq 1_G$, and $\calI_G^{1_G}=\calI_G=1/\#G$ (i.e.\ 
the constant function with value $1/\#G$ in $\C^G$).
By Proposition~\ref{prop:rho-comp2}, we have
\begin{eqnarray*}
I_G(f)-I_Y(f)&=&\left<f, \calI_G - \calI_Y \right> \\
&=& \sum_{\rho}\left<f_\rho, \calI_G^\rho - \calI_Y^\rho \right> \\
&=& \left<f_{1_G}, \calI_G - \calI_Y \right> + 
\sum_{\rho \neq 1_G} \left<f_\rho, - \calI_Y^\rho \right> \\
&=&\sum_{\rho \neq 1_G} \left<f_\rho, - \calI_Y^\rho \right>.
\end{eqnarray*}
(The last equality follows from $I_G(f)=I_Y(f)$ for a constant 
function $f$, and that for any $f$, $f_{1_G}$ is a constant function).
Thus
\begin{eqnarray*}
|I_G(f)-I_Y(f)|
&\leq&
\sum_{\rho \neq 1_G} |\left<f_\rho, - \calI_Y^\rho \right>| \\
&\leq&
\sum_{\rho \neq 1_G} ||f_\rho||\cdot ||\calI_Y^\rho||.
\end{eqnarray*}
Then, by
$\calI_Y^\rho=\frac{1}{\#Y}\delta_Y^\rho$
and using Propositions~\ref{prop:rho-comp} and \ref{prop:rho-comp2}
\begin{eqnarray*}
||\calI_Y^\rho||^2
&=&
\left<\calI_Y^\rho, \calI_Y^\rho \right> \\
&=&
\left<\calI_Y^\rho, \calI_Y \right> \\
&=& I_Y(\calI_Y^\rho)= \frac{1}{\#Y} \sum_{y \in Y}\calI_Y^\rho(y) \\
&=& \frac{1}{\#Y^2} \sum_{y \in Y}\sum_{x \in Y}\delta_x^\rho(y) \\
&=& \frac{\dim \rho}{\#Y^2\#G}
 \sum_{y,x \in Y}\chi_\rho(xy^{-1}) \\
&=& \frac{\dim \rho}{\#Y^2\#G}
 \sum_{x,y \in Y}{\chi_\rho(x^{-1}y)}.
\end{eqnarray*}
Hence, by putting 
$$
 \partial_\rho(Y):=||\calI_Y^\rho||=
 \sqrt{\frac{\dim \rho}{\#Y^2\#G}
 \sum_{x,y \in Y}{\chi_\rho(x^{-1}y)}},
$$
we prove Theorem~\ref{Th:main1}.
\begin{remark}\label{rem:equality}
If we put $f:=\calI_Y$, then the equalities hold
for the two inequalities in the proof. This is a worst
function for the Quasi-Monte Carlo integration by $Y$.
Moreover, if $Y$ is a pre-difference set, then
(4) of Theorem~\ref{Th:main3} implies that the equalities
hold in the inequalities in Corollary~\ref{sec:proof-th1}.
Thus the bound $V(f)D(Y)$ is tight in this sense.
\end{remark}
\section{Proofs of Proposition~\ref{prop:main2}
 and Theorem~\ref{Th:main3}}\label{sec:main3}
To obtain a lower bound of $D(Y)$,
we begin with a relation among 
$\partial_\rho(Y)$:
\begin{proposition}\label{prop:square}
For $\partial_\rho(Y)$ defined in Theorem~\ref{Th:main1},
$$
\sum_{\rho} \partial_\rho(Y)^2 = \frac{1}{\#Y}
$$
and $\partial_{1_G}(Y)^2=\frac{1}{\#G}$ for the trivial representation $1_G$.
Thus, 
$$
\sum_{\rho \neq 1_G} 
\partial_\rho(Y)^2 = \frac{1}{\#Y}-\frac{1}{\#G}.
$$
\end{proposition}

\begin{proof}
Since $\partial_\rho(Y) = ||\calI_Y^\rho||$,
\begin{eqnarray*}
\sum_{\rho} \partial_\rho(Y)^2
&=& \sum_{\rho} ||\calI_Y^\rho||^2 =
   || \sum_{\rho} \calI_Y^\rho ||^2 = ||\calI_Y||^2 \\
&=& \left<\calI_Y, \calI_Y\right> =\frac{\#(Y\cap Y)} {\#Y\#Y}
    = \frac{1}{\#Y}
\end{eqnarray*}
holds by a formula in Definition~\ref{def:I}.
For $\rho=1_G$, $\calI_Y^{1_G}$ is a constant function
with value $I_G(\calI_Y)=I_G(\frac{1}{\#Y}\delta_Y)=1/\#G$.
Thus 
$$||\calI_Y^{1_G}||^2=\sum_{g \in G}(1/\#G)^2=\frac{1}{\#G}.$$
\end{proof}

\begin{lemma}
Let $y_1,\ldots,y_n$ be non negative real numbers
with $\sum_{i} y_i = a$. Let $d_1,\ldots,d_n$ be
positive real numbers. Let $M$ be $\max_{i=1}^n y_i/d_i$.
Then $M \geq a/(\sum_{i}d_i)$, and the equality 
follows if and only if
$y_i/d_i= a/(\sum_{i}d_i)$ for each $i=1,\ldots,n$.
\end{lemma}
\begin{proof}
Since $M\geq y_i/d_i$,
$\sum_{i}(M d_i) \geq \sum_{i}y_i =a$ and hence the inequality
holds. If the equality holds, then $M=y_i/d_i$ holds
for each $i$, hence the result.
\end{proof}

Now we prove Proposition~\ref{prop:main2}. 
By Proposition~\ref{prop:square}, it follows that
$$
\sum_{\rho\neq 1_G} \partial_\rho(Y)^2 = \frac{1}{\#Y} - \frac{1}{\#G}. 
$$
By the lemma above with $y_i$ being $\partial_\rho(Y)^2$
and $d_i$ being $(\dim \rho)^2$, it follows that
$$
D(Y)^2 =
\max_{\rho \neq 1_G} (\partial_\rho(Y)/\dim \rho)^2 \geq 
\left(\frac{1}{\#Y} - \frac{1}{\#G}\right)/(\sum_{\rho \neq 1_G}(\dim \rho)^2),
$$
where $\sum_{\rho \neq 1_G}(\dim \rho)^2=\#G-1$ by Proposition~\ref{prop:group},
and the equality holds if and only if
$\partial_Y^\rho/\dim \rho=\sqrt{(\frac{1}{\#Y} - \frac{1}{\#G})/(\#G-1)}$
for any $\rho \neq 1_G$. This proves Proposition~\ref{prop:main2}.


\paragraph{Proof of Theorem~\ref{Th:main3}}

From definition of $\partial_\rho(Y)$ and Proposition~\ref{prop:rho-comp}
we have
\begin{eqnarray*}
\left(\frac{\partial_\rho(Y)}{\dim \rho}\right)^2
&=& 
\frac{1}{\dim \rho^2}
\sum_{x,y \in Y}
\frac{\dim \rho}{\#Y^2\#G}\chi_\rho(x^{-1}y)
\\
&=&
\frac{1}{\#Y^2\#G}\sum_{x,y \in Y}
\frac{\chi_\rho(x^{-1}y)}{\dim \rho}
\\
&=&
\frac{1}{\#Y^2\#G}\sum_{[a] \in C(G)}
\lambda_{[a]} \frac{\chi_\rho(a)}{\dim \rho}
\\
&=&
\frac{1}{\#Y^2\#G}\sum_{[a] \in C(G)}
\frac{\lambda_{[a]}}{\#[a]} \frac{\chi_\rho(a)}{\dim \rho}\#[a].
\end{eqnarray*}

Summarizing the above computation, we have:
\begin{lemma}\label{lem:invertible}
Let $(P(\rho, [a]))$ $(\rho \in \Ghat, [a]\in C(G))$
be the matrix defined by 
$P(\rho, [a])=\frac{1}{\#Y^2\#G}\frac{\chi_\rho(a)}{\dim \rho}\#[a]$.
By the orthogonality 
of the characters, this matrix is invertible. The above formula 
shows that 
the vector $(\frac{\lambda_{[a]}}{\#[a]})$ $([a] \in C(G))$
is mapped to 
the vector $(\frac{\partial_\rho(Y)}{\dim \rho})^2$ 
$(\rho \in \Ghat)$
by $P$.
\end{lemma}
We shall prove the equivalence between (1)-(5) in
Theorem~\ref{Th:main3}.

\noindent
Proof of $(1)\Leftrightarrow(2)$.
Now we assume (1).
Then $\lambda_{[a]}/\#[a]=\lambda$ 
for $a \neq 1$. 
Since $\sum_{[a]\in C(G)} \lambda_{[a]} = \#Y^2$,
$\sum_{[a]\in C(G)\setminus\{[1]\}} \lambda_{[a]} = \#Y^2-\#Y$
follows.
By $\lambda_{[a]}=\#[a]\lambda$, 
$\sum_{[a]\in C(G)\setminus\{[1]\}} \lambda_{[a]} = \lambda
\sum_{[a]\in C(G)\setminus\{[1]\}} \#[a] =\lambda(\#G-1)$.
Thus, $\lambda=\frac{\#Y^2-\#Y}{\#G-1}$, and (2) follows
since $\lambda_{[1]}=\#Y$ always hold.
Clearly (2) implies (1).

\noindent
Proof of $(2)\Leftrightarrow(3)$.
We show that (2) implies (3). 
Put $\lambda:=\frac{\#Y^2-\#Y}{\#G-1}$.
Lemma~\ref{lem:invertible} gives
for $\rho \neq 1_G$
\begin{eqnarray*}
\left(\frac{\partial_\rho(Y)}{\dim \rho}\right)^2
&=& 
\frac{1}{\#Y^2\#G}\sum_{[a] \in C(G)}
\frac{\lambda_{[a]}}{\#[a]} \frac{\chi_\rho(a)}{\dim \rho}\#[a] \\
&=&
\frac{1}{\#Y^2\#G}(\lambda_{[1]}\frac{\chi_\rho(1)}{\dim \rho}\#[1] + 
\sum_{[a] \neq [1]}
\lambda \frac{\chi_\rho(a)}{\dim \rho}\#[a]) \\
&=&
\frac{1}{\#Y^2\#G}(\#Y - 
\lambda \frac{\chi_\rho(1)}{\dim \rho}) \\
&=&
\frac{1}{\#Y^2\#G}(\#Y - \lambda)\\
&=&
\frac{1}{\#Y^2\#G}(\#Y - \frac{\#Y^2-\#Y}{\#G-1})\\
&=&
\frac{1}{\#Y^2\#G}\frac{\#Y\#G-\#Y^2}{\#G-1}\\
&=&
\frac{1/\#Y-1/\#G}{\#G-1}.
\end{eqnarray*}
The third equality used the orthogonality of 
the characters 
$$\sum_{a\in G}\overline{\chi_{1_G}(a)}\chi_\rho(a)=0$$
for $\rho \neq 1_G$ (\cite[I, \S2, Proposition~7]{serre}).
For $\rho=1_G$, a similar computation using Lemma~\ref{lem:invertible} gives
$\left(\frac{\partial_\rho(Y)}{\dim \rho}\right)^2=1/\#G$.
This implies (3), and that $P$ maps the vector $(v_{[a]})$
$([a] \in C(G))$ with components
$v_{[a]}=(\#Y^2-\#Y)/(\#G-1)$ ($[a]\neq [1]$) and
$v_{[1]}=\#Y$ to the vector $(w_{\rho})$ $(\rho \in \Ghat)$
with $w_\rho=\frac{1/\#Y-1/\#G}{\#G-1}$ $(\rho \neq 1_G)$
and $w_{1_G}=1/\#G$. Now assume (3). Since $P$ is a 
regular matrix, the above fact shows that (3) implies (2).

\noindent
Proof of $(3)\Leftrightarrow(4)$.
Clearly (3) implies (4). We assume (4), namely, 
$(\frac{\partial_\rho(Y)}{\dim\rho})^2=C$ for $\rho \neq 1_G$.
For $\rho=1_G$, 
$\frac{\partial_{1_G}(Y)}{\dim 1_G}^2=\frac{1}{\#G}$
is proved in Proposition~\ref{prop:square}, and
since $\partial_\rho(Y)^2=C(\dim \rho)^2$,
their sum over $\rho \neq 1_G$ is 
$$C \sum_{\rho \neq 1_G} (\dim \rho)^2=C(\#G-1).$$
By Proposition~\ref{prop:square}, this value is 
$\frac{1}{\#Y}-\frac{1}{\#G}$, and hence
$C=(\frac{1}{\#Y}-\frac{1}{\#G})/(\#G-1)$. 
This implies (3). 

\noindent
Proof of $(3)\Leftrightarrow(5)$.
By Proposition~\ref{prop:main2}, the equality of (5) holds
if and only if (3) holds, since 
$\partial_{1_G}(Y)=\sqrt{\frac{1}{\#G}}$ always holds.

These prove Theorem~\ref{Th:main3} (and thus Theorem~\ref{Th:main2}). 

\section{Pre-difference sets}\label{sec:pre}
\subsection{A pre-difference set which is not a difference set}
\begin{proposition}
Let $D_{16}$ be the dihedral group of order 16 presented by 
$$\left<s,r \mid s^2=r^8=(sr)^2=1\right>.$$
Then, the subset 
$$Y=\{1,r,s,sr^3,sr^5,sr^7\}$$
is a $(16,6,2)$ pre-difference set.
This is not a difference set.
\end{proposition}
A proof is done by a program using GAP\cite{GAP}.
The element $sr$ is not a difference of two elements in $Y$,
and hence $Y$ is not a difference set.

\begin{remark}
It is conjectured that any dihedral group $D_{2n}$ has
only trivial difference sets\cite[Remark 4.4]{DIFF-SET-FAN}.
The conjecture is true for $n$ being a prime power
(\cite[Theorem~1.3]{DIFF-SET-DENG} for the prime $2$,
\cite[Theorem~4.5]{DIFF-SET-FAN} for odd primes).
Thus, $D_{16}$ has no non-trivial difference set, but
has non-trivial pre-difference set.
\end{remark}

\subsection{Groups of order 16}
In this section, we consider 
only non-trivial pre-difference sets
and non-trivial difference sets.  
There are 14 isomorphism classes among groups of order 16.
A complete list of these classes and all the difference sets
there (up to equivalence stated below) is given in \cite{KIBLER}.
Let $G$ be a group and $Y\subset G$. Then,
$(\sigma,g)\in \Aut(G)\ltimes G$ acts on $Y$ by
$g\cdot\sigma(Y)$. If $Y$ is a pre-difference set
(a difference set, respectively), then so is
$g\cdot\sigma(Y)$, respectively. 
Such a (pre-)difference set is said to be equivalent
to $Y$.
Using GAP, we obtained a complete list
of pre-difference sets for these groups, as follows.
Among the 14 classes, the cyclic group and 
the dihedral group have no difference set. 
Accordingly in \cite{KIBLER}, the rest 12 classes are tagged
by (A) to (L). Among them, (A) to (D) are abelian, 
so we omit them since pre-difference sets in these groups are 
differece sets classified there.
We add the dihedral group as the class (M).
We list all the pre-difference sets in each class up to 
equivalence. If the pre-difference sets are difference sets,
we noted so. For each class, only the relators are described.
For example, in (E), the described group is generated by $a,b,c$
and the relations are 
$1=a^4=b^2=c^2=(ab)^2=aca^{-1}c^{-1}=bcb^{-1}c^{-1}$.
	\begin{enumerate}[label=(\Alph*), start=5]
		\item $a^4, b^2, c^2, (ab)^2, aca^{-1}c^{-1}, bcb^{-1}c^{-1}$\quad ($G \cong D_8\times \mathbb Z/2\mathbb Z$)
		\begin{enumerate}[label=\arabic*.]
			\item $1, a, b, c, a^2, abc$,
			\item (difference set) $1, a, b, c, a^{-1}, a^2bc$,
			\item (difference set) $1, a, b, c, a^2b, a^{-1}c$.
		\end{enumerate}
		
		\item $a^4, c^2, a^2b^{-2}, bab^{-1}a^{-3}, aca^{-1}c^{-1}, bcb^{-1}c^{-1}$\quad ($G\cong $Central product of $Q_8$ and $\mathbb Z/4\mathbb Z$)
		\begin{enumerate}[label=\arabic*.]
			\item (difference set) $1, a, b, c, a^2, abc$
			\item (difference set) $1, a, b, c, a^{-1}, b^{-1}c$
		\end{enumerate}

		\item $a^4, a^2b^{-2}, a^2c^{-2}, bab^{-1}a^{-3}, aca^{-1}c^{-1}, bcb^{-1}c^{-1}$\quad ($G\cong $Central product of $D_8$ and $\mathbb Z/4\mathbb Z$)
		\begin{enumerate}[label=\arabic*.]
			\item $1, a, b, c, a^2, abc$,
			\item $1, a, b, c, a^{-1}, bc^{-1}$,
			\item (difference set) $1, a, b, c, ab, c^{-1}$,
			\item (difference set) $1, a, b, c, ac, b^{-1}$.
		\end{enumerate}
		\item $a^4, b^2, aba^{-1}b^{-1}, c^2b^{-1}, cac^{-1}b^{-1}a^{-3}$\quad ($G \cong \mathrm{SmallGroup}(16, 3)$)
		\begin{enumerate}[label=\arabic*.]
			\item $1, a, c, a^2b, a^2, ac$,
			\item $1, a, c, a^2b, a^2, a^{-1}c^{-1}$,
			\item (difference set) $1, a, c, a^2b, a^{-1}, c^{-1}$,
			\item (difference set) $1, a, c, a^2b, a^2c, ab$,
			\item (difference set) $1, a, c, a^2, a^{-1}b, c^{-1}$,
			\item (difference set) $1, a, c, a^2, a^2c^{-1}, ab$.
		\end{enumerate}
		\item $a^4, b^4, bab^{-1}a^{-3}$\quad ($G\cong $Nontrivial semidirect product of $\mathbb Z/4\mathbb Z$ and $\mathbb Z/4\mathbb Z$)
		\begin{enumerate}[label=\arabic*.]
			\item $1, a, b, a^2, b^2, ab^{-1}$,
			\item (difference set) $1, a, b, a^2, ab^2, a^2b^{-1}$,
			\item (difference set) $1, a, b, a^2, b^{-1}, a^{-1}b^2$,
			\item (difference set) $1, a, b, b^2, a^2b, a^{-1}b^2$.
		\end{enumerate}
		\item $a^8, a^2b^{-2}, bab^{-1}a^{-5}$\quad ($G \cong M_{16}$)
		\begin{enumerate}[label=\arabic*.]
			\item $1, a, b, a^2, a^4, ab$,
			\item $1, a, b, a^2, a^3, aba^{-1}$,
			\item $1, a, b, a^4, a^3, b^{-1}$,
			\item $1, a, b, ab, a^3, a^{-3}$,
			\item $1, a, b, ab, a^{-3}, a^{-1}$,
			\item $1, a, b, a^3, aba^{-1}, a^{-2}$,
			\item (difference set) $1, a, b, a^4, a^3, a^2b$,
			\item (difference set) $1, a, b, a^4, a^{-1}, b^{-1}$.
		\end{enumerate}
		\item $a^8, a^4b^{-2}, bab^{-1}a^{-3}$\quad ($G \cong SD_{16}$)
		\begin{enumerate}[label=\arabic*.]
			\item $1, a, b, a^{-2}, b^2, ab$,
			\item $1, a, b, a^{-2}, b^2, a^{-1}b^{-1}$,
			\item $1, a, b, a^{-2}, b^2, ab^{-1}$,
			\item $1, a, b, a^{-2}, b^2, a^{-1}b$,
			\item $1, a, b, a^{-2}, a^{-1}, b^{-1}$,
			\item $1, a, b, b^2, a^{-1}, a^2b$,
			\item $1, a, b, b^2, a^2b, a^3$,
			\item (difference set) $1, a, b, a^{-2}, ab^2, a^2b$,
			\item (difference set) $1, a, b, ab^2, a^2b, a^2$.
		\end{enumerate}
		\item $a^8, a^4b^{-2}, bab^{-1}a^{-7}$\quad ($G\cong Q_{16}$)
		\begin{enumerate}[label=\arabic*.]
			\item $1, a, b, a^2, b^2, ab^{-1}$,
			\item (difference set) $1, a, b, a^2, ab^2, a^2b^{-1}$,
			\item (difference set) $1, a, b, b^2, a^3, a^2b^{-1}$.
		\end{enumerate}
		\item $a^8, b^2, abab$\quad $(G = D_{16})$
		\begin{enumerate}[label=\arabic*.]
			\item $1, a, b, a^2, a^4, a^{-3}b$,
			\item $1, a, b, a^2, a^{-3}, a^{-2}b$,
			\item $1, a, b, a^4, a^3, a^{-2}b$.
		\end{enumerate}
	\end{enumerate}


\subsection{Properties of pre-difference set}
The following theorem shows that pre-difference sets
share many properties of difference sets.
\begin{theorem}
Let $G$ be a finite group and $Y$ be a $(v,k,\lambda)$
pre-difference set. Then we have the following:
\begin{enumerate}
\item $\lambda$ is a positive integer.
\item $\lambda(v-1)=(k^2-k)$ holds.
\item If $Y\neq G$, then the complement $Y^c$
is a $(v,v-k,v-2k+\lambda)$ pre-difference set.
\end{enumerate}
\end{theorem}
\begin{proof}
(1). For any $a\in G$ with $a\neq 1$, $\lambda\#[a]$ is a positive
integer by definition. Now the conjugacy class formula
tells that
$$
\#G = 1 + \sum_{[a]\in C(G)\setminus \{[1]\}}\#[a].
$$
Multiply the both sides by $\lambda$. 
Since $\#G$ is a multiple of any $\#[a]$, $\#G\lambda$
is an integer. The latter term of the right hand side,
when multiplied by $\lambda$, is also an integer. 
Hence, $1\times \lambda$ is an integer.

(2). This is equivalent to the condition (2) of
Theorem~\ref{Th:main3}.

(3) For a subset $Y\subset G$, by abuse of language, 
let $Y$ denote $\sum_{y\in Y} y \in \C[G]$,
and $Y^{-1}$ denote $\sum_{y\in Y}y^{-1}$.
For any subset $H\subset G$, $HG=(\#H)G$ holds in $\C[G]$.
Let $C$ be $C(G)$.
Let $f:\C[G] \to \C[C]$ be the induced linear map from 
$G \to C, \ a \mapsto [a]$.
Then, $Y$ being a pre-difference set is 
equivalent to that $f(Y Y^{-1})=f((k-\lambda)e + \lambda G)$,
where $e$ is the unit of $G$.
Put $Z=G-Y$. Then, in $\C[G]$, we have
\begin{eqnarray*}
Z Z^{-1} 
&=&
(G-Y)(G-Y^{-1})=G^2-GY-GY^{-1}+YY^{-1} \\
&=& (\#G-2\#Y)G+YY^{-1}.
\end{eqnarray*}
Thus 
\begin{eqnarray*}
f(Z Z^{-1})&=&(\#G-2\#Y)f(G)+f((k-\lambda)e + \lambda G) \\
&=& f((v-2k+\lambda)G + (k-\lambda)e)
=f(\lambda'G+(k'-\lambda')e)
\end{eqnarray*}
holds for $\lambda'=v-2k+\lambda$
and $k'=v-k$.
This shows that $Y^c$ is a $(v, v-k, v-2k+\lambda)$
pre-difference set.
\end{proof}

\section{Association schemes and Delsarte theory}\label{sec:as}

\subsection{Basic facts}
Let us recall the notion of commutative association schemes briefly.
See \cite{Bannai}\cite{DELSARTE} for details.
Let $X$ be a finite set. By $M(X;\C)$ we denote the matrix algebra
over $\C$, where the rows and columns are indexed by $X$.
Let $C$ be a finite set with a specified
element $i_0 \in C$. 
Let $R: X\times X \to C$ be a surjective function.
For $i\in C$, $R^{-1}(i) \subset X \times X$ gives
a square matrix $A_i \in M(X;\C)$, where $A_i(x,y)$ is $1$
if $R(x,y)=i$ and 0 otherwise. We assume that $A_{i_0}$ is
the identity matrix. For any $i$, we assume that 
there is an $i'$ such that $^tA_i=A_{i'}$.
The tuple $(R,X,C)$ is called an association scheme
if the linear span of $A_i(x,y)$ $(i\in C)$ over $\C$
in the matrix algebra $M(X;\C)$ is closed under matrix
multiplication, and hence a subalgebra of $M(X;\C)$.
This subalgebra is called the Bose-Mesner algebra $A_X$ 
of the association scheme. If it is commutative, then
the association scheme is said to be commutative.
In this subsection, we deal with only commutative association schemes.

The ($A_i$) is a linear basis of $A_X$. 
Hadamard product of $A, B \in M(X;\C)$ is defined
by the component-wise product $(A\circ B)(x,y)=(A(x,y)B(x,y))$.
The set $\{A_i \mid i \in C\}$ consists of the primitive idempotents of $A_X$
with respect to the Hadamard product.
It is known that $A_X$ is closed under transpose and complex conjugate.
Since $M(X;\C)$ acts on $\C[X]$, so does $A_X$. Since $A_X$ is 
commutative, we may simultaneously diagonalize all elements of 
$A_X$. That is, we have a set of common eigen vectors $e_k \in \C[X]$
consisting a basis.
The action of $A_X$ on $e_k$ gives a ring homomorphism
$A_X \to \C$. Different $e_k$ may give the same ring homomorphism,
so let $\Xhat$ be the set of different ring homomorphisms
$\rho: A_X \to \C$ obtained in this way. 
Then, $A_X \to \C^\Xhat$ is an isomorphism (where
the multiplication of $A_X$ is given by the matrix multiplication).
Thus, there is a set of primitive idempotents $E_\rho\in A_X$, $\rho \in \Xhat$.
There is a special primitive idempotent $E_{j_0}:=\frac{1}{\#X}J$, where $J$ denotes
the matrix with all components being 1.
The corresponding representation $j_0:A_X \to C$ is given by 
$A_iE_{j_0}=j_0(A_i)E_{j_0}$. 
Thus $j_0(A_i)$ is the number of ones in a column in $A_i$
(independent of the choice of the column), 
called the $i$-th valency and denoted by $k_i$. We call $j_0$
the trivial representation.

This shows that
$$ \C^C \to A_X \to \C^\Xhat $$
are isomorphisms of $\C$-vector spaces, where 
the left map is an isomorphism as a ring ($\C^C$ the direct product
and $A_X$ the Hadamard product), and the right map is 
an isomorphism as a ring ($A_X$ the matrix product and $\C^\Xhat$ the direct
product).

We have orthogonal decomposition of $\C[X]=\oplus_{\rho\in \Xhat}V_\rho$, where
$V_\rho$ is the largest subspace 
such that $A_X$ acts on $V_\rho$ via character
$\rho$ of $A_X$. 
The one dimensional subspace spanned by 
$\sum_{x \in X} x \in \C[X]$ is $V_{j_0}$.

A typical example of commutative association schemes 
is a group association scheme associated to a finite group $G$.
In this case, $X=G$, $C=C(G)$, and 
$R: G \times G \to C(G)$ is given by $x, y \mapsto [x^{-1}y]$.
The group ring $\C[G]$ acts from left on $\C[G]$, 
and thus $\C[G] \subset M(G;\C)$. It is known that
the Bose-Mesner algebra $A_X$ is the center of $\C[G]$
($=\C[C(G)]$), $A_{[a]}$ is the matrix representation of 
$[a] \in \C[G]$. 
The set $\Xhat$ may be taken as $\Ghat$,
and $E_\rho$ is the projection from $\C[G]$ to the $\rho$-component
of $\C[G]$ (as a left $G$-module).

\subsection{Quasi-Monte Carlo in an association scheme}
Let $Y\subset X$ be a non-empty finite set. 
Objects defined in Definition~\ref{def:set} such as 
$\delta_Y:=\sum_{y\in Y}y \in \C[X]$ are available for any set. 
Take $f \in \C[X]$, 
and let $\calI_Y$ be $\frac{1}{\#Y}\delta_Y$.
Then $\left<f, \calI_Y\right>=\frac{1}{\#Y}\sum_{y\in Y}f(y)$.
Any $f$ decomposes to the sum of $f_\rho \in V_\rho$
uniquely, and $f_\rho$ is called the $\rho$-component
of $f$. (For $f=I_Y$, its $\rho$-component is 
denoted by $I_Y^\rho$). 
Because of the orthogonality of $V_\rho$, 
$||f||^2=\sum_{\rho}||f_\rho||^2$.
Because $\C \cdot \calI_X \subset \C[X]$ is the $j_0$
component of $\C[X]$, 
$\calI_X^\rho = 0$ for $\rho \neq j_0$.
Generally, $f_{j_0}=\sum_{x\in X} \left<f,\calI_X\right> x$ holds.

By defining $\partial_\rho(Y):=||\calI_Y^\rho||$,
we obtain the same inequalities as in Theorem~\ref{Th:main1}.
\begin{theorem}\label{Th:main1-as}
Let $(R,X,C)$ be a commutative association scheme,
$Y$ a nonempty subset of $X$,
and $f:X \to \C$ a function. This function
is identified with $\sum_{x \in X}f(x)x \in \C[X]$.
Let $I(f)$ be the
average of $f$ over $X$, $I_Y(f)$ the average of $f$
over $Y$. Let $f_\rho$ be the $\rho$-component of 
$f$ for each $\rho \in \widehat{X}$.
Define 
$
 \partial_\rho(Y):=||\calI_Y^\rho||.
$
Then we have
$|I(f_\rho)-I_Y(f_\rho)| \leq ||f_\rho||\partial_\rho(Y)$.
For the trivial representation $\rho=j_0$, we have
$|I(f_{j_0})-I_Y(f_{j_0})|=0$.
\end{theorem}
\begin{proof}
This is because 
$$
I(f_\rho)-I_Y(f_\rho)=\left<f_\rho, \calI_X-\calI_Y\right>,
$$
$(\calI_X^{j_0}-\calI_Y^{j_0})=0$ and $\calI_X^{\rho}=0$ for 
$\rho \in \widehat{X}, \rho \neq j_0$.
The right hand side $\left<f_\rho, \calI_X-\calI_Y\right>$ 
of the equation for $\rho\neq j_0$ is
$$
\left<f_\rho, \calI_X^\rho-\calI_Y^\rho\right> 
 =  \left<f_\rho, -\calI_Y^\rho\right> \\
$$
and its absolute value is bounded by 
$$
||f_\rho||\cdot ||\calI_Y^\rho||
=
||f_\rho||\cdot \partial_\rho(Y).
$$
\end{proof}

\begin{corollary}\label{cor:weight-as}
We have 
$$
|I(f)-I_Y(f)| \leq 
\sum_{\rho \in \widehat{X} \setminus\{j_0\}}||f_\rho||\partial_\rho(Y)
\leq  V(f)D(Y),
$$
where
$V(f):=\left(\sum_{\rho\in \widehat{X} \setminus\{j_0\}}
||f_\rho||\sqrt{\dim V_\rho} \right)$
and  
$D(Y):=\max_{\rho \in \widehat{X}\setminus \{j_0\}} 
\frac{\partial_\rho(Y)}{\sqrt{\dim V_\rho}}$.
\end{corollary}
The proof is the same as that of Corollary~\ref{cor:weight}.

\begin{proposition}\label{prop:square-as}
For $\partial_\rho(Y)$ defined in Theorem~\ref{Th:main1-as},
$$
\sum_{\rho} \partial_\rho(Y)^2 = \frac{1}{\#Y}
$$
and $\partial_{j_0}(Y)^2=\frac{1}{\#X}$.
Thus, 
$$
\sum_{\rho \neq j_0} 
\partial_\rho(Y)^2 = \frac{1}{\#Y}-\frac{1}{\#X}.
$$
\end{proposition}
\begin{proof}
The first equality follows from
$$
\sum_{\rho} \partial_\rho(Y)^2 = 
\sum_{\rho} ||\calI_Y^\rho||^2 = 
\left<\calI_Y, \calI_Y\right>=1/\#Y.
$$
(See the formula in Definition~\ref{def:I}.)
We have
$$
\partial_{j_0}(Y)^2 =\langle \calI_Y^{j_0}, \calI_Y^{j_0}\rangle,
$$
but this is $1/\#X$ since  
$
\calI_Y^{j_0} (x)=  \langle \calI_Y,\calI_X \rangle
=\frac{1}{\#X}\frac{1}{\#Y}(\#(Y\cap X))
=\frac{1}{\#X}
$
and hence 
$\langle \calI_Y^{j_0}, \calI_Y^{j_0}\rangle =\#X\frac{1}{\#X^2}
=\frac{1}{\#X}.
$
\end{proof}

\begin{proposition}\label{prop:main2-as}
Let $D(Y)$ be the 
value defined in Corollary~\ref{cor:weight-as}.
Under the assumptions in Theorem~\ref{Th:main1-as},
a lower bound 
$$
D(Y) \geq 
\sqrt{\frac{1/\#Y-1/\#X}{\#X-1}}
$$
holds.
The equality holds 
if and only if 
$\frac{\partial_\rho(Y)}{\sqrt{\dim V_\rho}} 
= \sqrt{\frac{1/\#Y-1/\#G}{\#G-1}}$
holds for any $\rho \neq j_0$.
\end{proposition}
The proof is the same as that of Proposition~\ref{prop:main2}.

\subsection{Delsarte theory and difference sets}
\begin{definition}
The $\C$-vector space $M(X;\C)$ is equipped
with the standard Hermitian inner product
$$
\left< A, B \right> = \trace (A B^*) 
= \sum_{x,y \in X} a_{x,y}\overline{b_{x,y}}.
$$
\end{definition}

\begin{definition}
For $Y\subset X$, we define $\Delta_Y\in M(X;\C)$
by $(\Delta_Y)(x,y)=1$ for $x,y\in Y$ and $0$ otherwise.
\end{definition}

\begin{lemma}\label{lem:inner-distribution}
We define
$$\lambda_i(Y):=\left<A_i, \Delta_Y\right>.$$
Then $\lambda_i(Y)=\#\{(x,y)\in Y^2 \mid A_i(x,y)=1\}$.
{\rm(}This is the inner distribution in \cite{DELSARTE} 
 multiplied by a scalar $\#Y$.
If $X$ is a group association scheme of $G$, 
then $\lambda_{[a]}(Y)=\lambda_{[a]}$ defined in 
Theorem~\ref{Th:main3}.{\rm)}

We have
$$||\delta_Y^\rho||^2=\left<  E_\rho, \Delta_Y \right>.$$
\end{lemma}
\begin{proof}
For any $A \in M(X,\C)$, we have
$
 \left< A , \Delta_Y \right> = 
 \left<A\delta_Y, \delta_Y \right>
$
where the second inner product is that for $\C[X]$.
The first statement follows immediately, and 
the second statement follows from $\delta_Y^\rho=E_\rho \delta_Y$
and the orthogonality of $V_\rho$.
\end{proof}

We prepare for a generalization of Theorem~\ref{Th:main3}.
\begin{proposition}
Let $X,R:X\times X \to C$ be a commutative association scheme.
The Hadamard idempotents $A_i$ $(i \in C)$ form a
{\rm(}not necessary normalized{\rm)} orthogonal basis of $A_X$,
and $\left< A_i,A_i \right>=k_i \#X$. 

The ordinal idempotents $E_\rho$ $(\rho \in \Xhat)$
form also a {\rm(}not necessary normalized{\rm)} orthogonal basis of $A_X$,
and $\left<E_\rho, E_\rho \right>=\dim V_\rho$.
\end{proposition}
\begin{proof}
Orthogonality of $A_i$ comes from that 
$\left< A, B \right>$ 
is the sum of all components of the Hadamard product $A\circ \overline{B}$
and $A_i$'s are primitive idempotents. The value of inner product
is an easy counting.
Orthogonality of $E_\rho$ comes from 
$\left< A, B \right>=\trace AB^* $, $E_\rho^*=E_\rho$, 
and that $E_\rho$'s are primitive idempotents.
Since $E_\rho$ is the projector to $V_\rho$,
its trace is $\dim V_\rho$, which is 
$\left< E_\rho,E_\rho\right> = \trace E_\rho^2=\trace E_\rho$.
\end{proof}
The following lemma is obvious.
\begin{lemma}
Let $V$ be a $\C$-vector space with Hermitian inner product
$\left< , \right >$. Let $W\subset V$ be a subspace, with
an orthogonal basis $w_1,\ldots,w_n$.
Then the orthogonal projection $V \to W$, $v \mapsto v_W$
is given by 
$$
v_W = \sum_{i=1}^n \frac{\left< v, w_i\right>}{\left< w_i, w_i\right>}w_i.
$$
\end{lemma}
\begin{proof}
For a unique $h \in W^\perp$, we have
$$v=v_W \oplus h= \sum_{i} a_i w_i \oplus h.$$
Since 
$$\left<v, w_i\right > = a_i \left< w_i, w_i \right>,$$
it follows that
$$
 v_W = \sum_{i} a_i w_i 
= \sum_{i=1}^n \frac{\left< v, w_i\right>}{\left< w_i, w_i\right>}w_i.
$$
\end{proof}
\begin{corollary}
Let $X$ be a commutative association scheme.
For any $M \in M(X,\C)$, its orthogonal 
projection $M_A$ to $A_X$ is
\begin{eqnarray*}
M_A &=&
\sum_{i\in C} \frac{\left< M, A_i\right>}{\left< A_i, A_i\right>}A_i
\\
&=&
\sum_{\rho\in \Xhat} \frac{\left< M, E_\rho\right>}{\left< E_\rho, E_\rho\right>}E_\rho.
\end{eqnarray*}
\end{corollary}
Putting $M=\Delta_Y$, we proved
\begin{corollary}\label{cor:lambda-delta}
\begin{eqnarray*}
(\Delta_Y)_A
&=&
\sum_{i\in C} \frac{\lambda_i(Y)}{k_i\#X}A_i
\\
&=&
\sum_{\rho\in \Xhat} \frac{||\delta_Y^\rho||^2}{(\dim V_\rho)}
E_\rho.
\end{eqnarray*}
\end{corollary}
\begin{remark}
The vector $\lambda_i(Y)$ $(i \in C)$ is called the
inner distribution vector of $Y$ (times $\#Y$) in 
\cite[\S3.1 (3.1)]{DELSARTE}. The above result is 
proved implicitly in \cite[\S3.1]{DELSARTE}, but 
we give a proof here for simplicity and the self-containedness.
Note that the definition of inner distribution
adopted in \cite[Definition~4.1]{DELSARTE2}
coincides with $\lambda_i(Y)/k_i$. The above
result is also deduced from the arguments in Section~3 there.
\end{remark}

\begin{lemma}
Let $V$ be the linear subspace of $A_X$
spanned by the identity matrix $I$ and $J$.
Then, $\sum_{i\in C}a_iA_i \in V$
if and only if the value $a_i$ is independent
of the choice of $i \in C \setminus \{i_0\}$.
Similarly, 
$\sum_{\rho \in \Xhat}a_\rho E_\rho \in V$
if and only if the value $a_\rho$ is independent
of the choice of $\rho \in \Xhat \setminus \{j_0\}$.
\end{lemma}
\begin{proof}
It is known that $A_{i_0}=I$ and $\sum_{i\in C}A_i=J$.
The first statement follows from the linear independence
of $A_i$. Also, it is known that $E_{j_0}=\frac{1}{\#X}J$
and $\sum_{\rho\in \Xhat} E_\rho=I$. The second statement
follows from the linear independence of $E_\rho$.
\end{proof}

The next theorem is a direct consequence of
the above lemma and Corollary~\ref{cor:lambda-delta}.
\begin{theorem}\label{Th:assoc-main}
Let $X$ be a commutative association scheme, and $Y$ its subset.
Then, $\frac{\lambda_i(Y)}{k_i\#X}$ is independent 
of the choice of $i \neq i_0$ if and only if
$\frac{||\delta_Y^\rho||^2}{(\dim V_\rho)}$
is independent of the choice of $\rho \neq j_0$.
\end{theorem}
This theorem and Proposition~\ref{prop:main2-as} 
show the following theorem,
which generalizes Theorem~\ref{Th:main3}.
\begin{theorem}\label{Th:main3-as}
Let $(X,R,C)$ be a commutative association scheme
and $Y$ a non-empty subset of $X$.
Then, the following conditions
are equivalent.
\begin{enumerate}
\item $\lambda_i(Y)/k_i$ is independent of choice of $i\in C$
except $i =i_0$, where $k_i$ is the number of 1 
in a column of $A_i$.
\item
$\lambda_i(Y)/k_i=(\#Y^2-\#Y)/(\#X-1)$
holds for any $i \neq i_0$, and $\lambda_{i_0}=\#Y$ holds.
\item
$
 \partial_\rho(Y)/\sqrt{\dim V_\rho}=
 \sqrt{\frac{1/\#Y-1/\#X}{\#X-1}}
$
holds for any $\rho \neq j_0$, and
$\partial_{j_0}(Y)=\sqrt{\frac{1}{\#X}}$ holds.
\item
$\partial_\rho(Y)/\sqrt{\dim V_\rho}$ is independent of 
the choice of $\rho \in \widehat{X}\setminus\{j_0\}$.
\item 
The equality holds in the inequality
$$
D(Y)
\geq 
\sqrt{\frac{1/\#Y-1/\#X}{\#X-1}}
$$
in Proposition~\ref{prop:main2-as}.
\end{enumerate}
\end{theorem}
\begin{proof}
Recall that $\partial_\rho(Y)=||\calI_Y^\rho||$.

\begin{itemize}
\item
Equivalence $(1)\Leftrightarrow(4)$:
Theorem~\ref{Th:assoc-main} implies the equivalence 
between (1) and (4). 

\item
Equivalence $(3)\Leftrightarrow(5)$:
Proposition~\ref{prop:main2-as}
implies the equivalence between (3) and (5), since
$\partial_{j_0(Y)}=\sqrt{1/\#X}$ always hold.

\item
Equivalence $(3)\Leftrightarrow(4)$:
(3) implies (4). Assume (4), 
and put $K:=\partial_\rho(Y)^2/(\dim V_\rho)$.
Proposition~\ref{prop:square-as} implies $\partial_{j_0}(Y)^2=1/\#X$.
Putting $\partial_\rho(Y)^2=K(\dim V_\rho)$ in the second
equality in Proposition~\ref{prop:square-as},
we have
$\sum_{\rho \neq j_0}K(\dim V_\rho)=1/\#Y-1/\#X$.
Since $\sum_{\rho \neq j_0}\dim V_\rho=\#X-1$,
we have $K=(1/\#Y-1/\#X)/(\#X-1)$, which implies (3).

\item
Equivalence $(1)\Leftrightarrow(2)$:
(2) implies (1). Assume (1), and put $K:=\lambda_i(Y)/k_i$.
From Lemma~\ref{lem:inner-distribution}, it follows
that $\lambda_{i_0}(Y)=\#Y$, and 
$\lambda_i(Y)=\#(Y^2\cap R^{-1}(i))$.
Thus, $\sum_{i \in C}\lambda_i(Y)=\#Y^2$.
Putting $\lambda_i(Y)=k_iK$ for $i\neq i_0$
and using $\sum_{i\neq i_0} k_i = \#X-1$ (since $\sum_i A_i=J$),
we have $K(\#X-1)+\#Y=\#Y^2$, and thus
$K=(\#Y^2-\#Y)/(\#X-1)$, which implies (2).
Note that this proof does not use the commutativity of 
the association scheme.
\end{itemize}
\end{proof}
We give a second proof of Theorem~\ref{Th:main3}.
\begin{proof}
Let $G$ be a finite group, and consider the associated
commutative group
association scheme $(G,R,C(G))$. 
Let $Y$ be a non-empty subset of $G$.
Then $\frac{\lambda_{[a]}(Y)}{k_i}=\frac{\lambda_{[a]}}{\#[a]}$
and 
$\frac{||\delta_Y^\rho||^2}{(\dim V_\rho)}
=\frac{||\delta_Y^\rho||^2}{(\dim \rho)^2}$,
since $H_\rho$ has dimension $(\dim \rho)^2$.
Now the conditions (1)-(5) in Theorem~\ref{Th:main3}
are the same as those in Theorem~\ref{Th:main3-as}.
\end{proof}

We close this paper by proposing a notion of 
a {\em difference set in an association scheme},
which is equivalent to the condition (1) in Theorem~\ref{Th:main3-as}
and unifies the notions of difference set and pre-difference set.
\begin{definition}\label{def:diff-set-in-as}
Let $(X,R,C)$ be an association scheme, which may be 
non-commutative. Let $i_0\in C$ be the element
with $A_{i_0}=I$. A non-empty subset $Y$ of $X$ is said to be 
a difference set in the association scheme $(X,R,C)$, 
if there is a constant $\lambda \in \Q$ such that
$\lambda=\#(Y^2\cap R^{-1}(i))/k_i$ holds for any $i\in C$
except $i =i_0$, where $k_i$ is the valency of $A_i$.
Let $v:=\#X$, $k:=\#Y$. Then $Y$ is called a $(v,k,\lambda)$
difference set in $(X,R,C)$.
\end{definition}
It is known that for a finite group $G$, $R:G\times G \to G$
given by $R(g,h)=g^{-1}h$ is an association scheme $(G,R,G)$
(which is commutative if and only if $G$ is commutative).
A difference set $Y\subset G$ (in a usual sense) is a difference set in 
the association scheme $(G,R,G)$ with $Y\neq G$. 
A pre-difference set 
is a difference set in the group association scheme
$(G,R,C(G))$. 

\begin{proposition}
Let $(X,R,C)$ be an association scheme and $Y$ a difference set
as in Definition~\ref{def:diff-set-in-as}. Then,
\begin{enumerate}
 \item $\lambda(v-1)=k(k-1)$ holds.
 \item The complement $Y^c$ is a $(v,v-k,v-2k+\lambda)$
 difference set in $(X,R,C)$.
\end{enumerate}
It is clear that if $\#Y=1$ or $v$, $Y$ is a difference set.
By (2), so is $Y$ if $\#Y=v-1$.
We call these difference sets trivial.
\end{proposition}
\begin{proof}
Theorem~\ref{Th:main3-as}~{$(1)\Rightarrow (2)$} (whose proof does not
use the commutativity) shows that 
$\lambda(v-1)=k(k-1)$ holds. 

We use notations as in the proof of 
Lemma~\ref{lem:inner-distribution}. Then
$\lambda_i(Y)=\left<A_i\delta_Y, \delta_Y \right>$.
We have $\delta_{Y^c}=\delta_{X}-\delta_{Y}$, and
\begin{eqnarray*}
\lambda_i(Y^c) &=&
 \left<A_i(\delta_X-\delta_Y), \delta_X-\delta_Y\right> \\
&=&
 \left<A_i\delta_X, \delta_X\right>
-\left<A_i\delta_X, \delta_Y\right>
-\left<A_i\delta_Y, \delta_X\right>
+\left<A_i\delta_Y, \delta_Y\right> \\
&=&
k_i\#X - k_i \#Y -
\left<\delta_Y, ^tA_i\delta_X\right>
+\lambda_i(Y) \\
&=&
k_iv -2k_i k + \lambda_i(Y).
\end{eqnarray*}
By dividing by $k_i$, we have
$$
\lambda_i(Y^c)/k_i=v-2k+\lambda
$$
for every $i\neq i_0$, which proves the second statement. 
\end{proof}
\begin{remark}
Using a table of all association schemes of order 16 \cite{HANAKI}, 
we find examples of difference sets in association schemes
with non-integer value of $\lambda=2/5, 4/5, 4/3, 14/5$, etc.,
as well as usual integer values $\lambda=2, 6$, by using GAP.
\end{remark}
\begin{example} 
Let $k\leq n$ positive integers.
Let $[n]$ be $\{1,2,\ldots,n\}$, and $X$ the set of
all subsets of $[n]$ with cardinality $s$. For
$B_1, B_2 \in X$, we define $R(B_1,B_2):=s-\#(B_1\cap B_2)$.
Then, $(X, R, C)$ with $R:X\times X \to C:=\{0,1,\ldots,s\}$
is a symmetric (hence commutative) association scheme 
called a Johnson scheme, denoted by $J(n,s)$.

Difference sets $Y$ and $Y'$ in $J(n,s)$ is 
said to be equivalent if there is a bijection $\phi:X \to X$
with $\phi(Y)=\phi(Y')$ (note that $\phi$ acts on the
set of subsets of $X$).

In $J(5,2)$, we show that only the following
difference sets exist up to equivalence, with an
aid of a computer. 
Let $k$ denote the cardinality of $Y$.
If $k=1,9$ or $10$, then every $Y$ is a trivial difference set.
If $k=2,5,8$, there are no difference sets.
If $k=3$, $\{\{1,2\},\{1,3\}, \{3,4\}\}$ is a 
difference set with $\lambda=2/3$.
If $k=4$, two difference sets
$\{\{1,2\},\{1,3\},\{3,4\},\{2,4\}\}$ and
$\{\{1,2\},\{1,4\},\{3,4\},\{4,5\}\}$ exist
with $\lambda=4/3$.
If $k=6$, there are two difference sets with $\lambda=10/3$, 
which are the complement sets of the $k=4$ cases.
If $k=7$, there is one difference set with $\lambda=14/3$,
which is the complement of the $k=3$ case.
These examples show that $\lambda$ may not be integers.
\end{example}
\appendix
\section{Quasi-Monte Carlo integration and characters}\label{sec:appendix}
Here we explain a typical example of QMC integration.
Put $X:=(\R/\Z)^s$.
We consider a periodic real valued function of $s$-variables
$f: X \to \R$. Let $\alpha$ be a positive integer, and
$\br=(r_1,\ldots,r_s)$. We write $\br\leq \alpha$
if $r_i\leq \alpha$ holds for each $1\leq i \leq s$.
Let $\bx=(x_1,\ldots,x_s) \in (\R/\Z)^s$.
For $f(\bx)$, we define
$$
D^\br(f):=
 \frac{\partial^{r_1+\cdots+r_s}}
 {\partial x_1^{r_1}\cdots \partial x_s^{r_s}}(f)
$$
if it exists. We assume that $D^\br(f)$ exists and is continuous
for $\br \leq \alpha$. We define a norm
$$
||f||_\alpha:=\sum_{0\leq \br \leq \alpha} ||D^\br(f)||_{L^\infty}.
$$
The set $\Xhat$ of characters of $X$ is 
$$
\Xhat = \{E_\bh(\bx):=\exp(2\pi i \bh \cdot \bx)
\ | \
\bh = (h_1,\ldots,h_s) \in \Z^s
\}.
$$
Let 
$$
f(\bx)=\sum_{\bh \in \Z^s}\hat{f}(\bh)E_\bh(\bx)
$$
be the Fourier-expansion of $f$. 
Note that $\hat{f}(0)=I(f):=\int_{X} f(\bx)$.
Let $P$ be a finite subset in $X$. The QMC integration 
of $f$ by $P$ is the the average
$$
I(f;P):=\frac{1}{\#P}\sum_{\bx \in P}f(\bx),
$$
and the integration error is
\begin{eqnarray}\label{eq:error}
\notag
\Err(f;P)&=&|I(f)-I(f;P)|
=\left| \hat{f}(0)-\sum_{\bh \in \Z^s}\hat{f}(\bh)I(E_\bh;P)\right| \\
&\leq &\sum_{\bh \in \Z^s-\{0\}}
 \left|\hat{f}(\bh)\right| \cdot \left|I(E_\bh;P)\right|.
\end{eqnarray}
Let 
$$
\rho(\bh):=\max\{1,|h_1|\}\times\cdots\times\max\{1,|h_s|\}.
$$
It is not difficult to show that the inequalities on the
Fourier coefficients
\begin{equation}\label{eq:decay-fourier}
|\hat{f}(\bh)|\leq C_{s,\alpha} ||f||_\alpha \rho(\bh)^{-\alpha}
\end{equation}
hold for a constant $C_{s,\alpha}$ depending only on $s,\alpha$
(cf. \cite[\S2.2]{Dick-MCQMC}\cite[\S.5.2.2]{GODA-SUZUKI-QMC-survey}),
and we have a Koksma-Hlawka type inequality on the error bound:
\begin{equation}\label{eq:KHtype}
\Err(f;P)\leq C_{s,\alpha}||f||_\alpha\times 
\sum_{\bh \in \Z^s-\{0\}}\left| \rho(\bh)^{-\alpha} I(E_\bh;P)\right|.
\end{equation}
Thus, we want a point set $P$ that makes 
$\sum_{\bh \in \Z^s-\{0\}}\left| \rho(\bh)^{-\alpha} I(E_\bh;P)\right|$
small. This is attained if $|I(E_\bh;P)|$ is small
(or ideally $0$) for $\bh$ with small $\rho(\bh)$. 

From now on, we assume that $P\subset X$ is a finite cyclic subgroup of
order $N$. 
Such a point set is called a rank-1 lattice and 
well-studied; see a nice introduction
\cite{SLOAN-JOE}.
Let $P^\perp \subset \Z^s \cong \Xhat$ be the subgroup defined by
$$
P^\perp:=\{\bh \in \Z^s \ | \ E_\bh(\bx)=1 \mbox{ for all $\bx \in P$}\}.
$$ 
It is easy to show that $I(E_\bh;P)=0$ if $\bh \notin P^\perp$,
and $I(E_\bh;P)=1$ if $\bh \in P^\perp$.
Thus, we obtain the error-bound
$$
\Err(f;P)\leq C_{s,\alpha}||f||_\alpha\times 
\sum_{\bh \in P^{\perp}-\{0\}}\left| \rho(\bh)^{-\alpha} \right|.
$$
It is known that for any $N$ there are $P$ such that 
the summation in the right hand side
is bounded by $C_{s,\alpha}'N^{-\alpha}(\log N)^{s\alpha}$
\cite[Chapter~5]{niederreiter:book}, which implies that
when $N$ is increased, the error-bound decreases with
order $O(N^{-\alpha}(\log N)^{s\alpha})$.

When compared with our study, 
the above error analysis (\ref{eq:error}) is essentially obtained
from the left inequality in Corollary~\ref{cor:weight}
(if we neglect that we treat only finite groups), 
where $\rho\in \Ghat$ corresponds to 
$E_\bh\in \Xhat$, $f_\rho$ corresponds to $\hat{f}(\bh)$,
$Y$ corresponds to $P$, and 
$
 \partial_\rho(Y)=||\calI_Y^\rho||=
 |\langle E_\rho, \calI_Y\rangle|= |I(E_\rho;Y)|
$
corresponds to $|I(E_\bh;P)|$. A big difference is that
as in (\ref{eq:decay-fourier}), 
$\hat{f}(\bh)$ (the $\bh$-component of $f$)
decays when $\rho(\bh)$ gets large, and to make
the error bound smaller, it is better to choose $P$
such that $|I(E_\bh;P)|=0$ holds for $\bh$ 
with small $\rho(\bh)$ since $I(E_\bh;P)$ has a large ``weight''
in the bound (\ref{eq:KHtype})
(this condition is close to the idea of Delsarte's ``designs''
in \cite[\S3.4]{DELSARTE}),
while in our study, as in Theorem~\ref{Th:main3}, 
we have no reasonable ``weight'' on characters 
(i.e. degree of importance of each character),
and we need to treat them with equal importance, which 
leads to the notion of pre-difference sets.
At present, we think that our results treating 
general groups and association schemes
are rather wide and abstract (say, compared with $(\R/\Z)^s$)
and that a practical application to QMC
is a future work.

We here briefly explain the notion of digital nets
in the theory of QMC \cite{niederreiter:book} \cite{DICK-PILL-BOOK},
which are widely used and closely related with character theory.
The unit hypercube $[0,1]^s$ is approximated by $(\F_2^n)^s$ 
via two-adic expansion upto $n$ digits. A digital net
$P$ is a subgroup of $(\F_2^n)^s$, identified as a subset
of the hypercube. 
An important figure-of-merit 
of $P$
is its $t$-value (\cite[Chapter~4]{niederreiter:book}),
which is obtained from the minimum weight of $P^\perp$
with respect to Niederreiter-Rosenbloom-Tsfasman(NRT) metric
\cite{NIEDERREITER-PIRSIC}, which is
a generalization of the Hamming weight. $P$ has a good (large)
$t$-value if and only if $I(E_\rho;P)=0$ holds for every
$\rho\in \widehat{(\F_2^n)^s}-\{0\}$ with small NRT metric,
which can be formalized by the notion of designs by Delsarte,
as mentioned above. See \cite{MARTIN-STINSON} for
analysis of the digital nets via association schemes.
Our study is different in that we treat the cases 
where the $\partial_\rho(Y)/\dim \rho=|I(E_\rho;Y)|/\dim \rho$
are independent of $\rho \neq 1_G$, while the above studies
treat the cases where $|I(E_\rho;Y)|=0$ holds for 
some ``important'' characters $\rho$.

\bibliographystyle{spmpsci.bst}
\bibliography{sfmt-kanren}

\begin{thebibliography}{10}
\providecommand{\url}[1]{{#1}}
\providecommand{\urlprefix}{URL }
\expandafter\ifx\csname urlstyle\endcsname\relax
  \providecommand{\doi}[1]{DOI~\discretionary{}{}{}#1}\else
  \providecommand{\doi}{DOI~\discretionary{}{}{}\begingroup
  \urlstyle{rm}\Url}\fi

\bibitem{Bannai}
Bannai, E., Ito, T.: Algebraic Combinatorics I: Association Schemes.
\newblock Benjamin / Cummings, Calfornia (1984)

\bibitem{difference-set}
Bruck, R.H.: Difference sets in a finite group.
\newblock Trans. Amer. Math. Soc. \textbf{78}, 464 -- 481 (1955)

\bibitem{DELSARTE}
Delsarte, P.: An algebraic approach to the association schemes of coding
  theory.
\newblock Philips Res. Rep. Suppl. \textbf{10}, i--vi and 1--97 (1973)

\bibitem{DELSARTE2}
Delsarte, P.: Pairs of vectors in the space of an association scheme.
\newblock Philips Res. Rep. Suppl. \textbf{32}, 373--411 (1977)

\bibitem{DIFF-SET-DENG}
Deng, Y.: A note on difference sets in dihedral groups.
\newblock Arch. Math. (Basel) \textbf{82}, 4--7 (2004)

\bibitem{Dick-MCQMC}
Dick, J.: On quasi-{Monte} {Carlo} rules achieving higher order convergence.
\newblock In: Monte Carlo and Quasi-Monte Carlo Methods 2008, pp. 73--96.
  Springer (2009)

\bibitem{DICK-KUO-SLOAN}
Dick, J., Kuo, F., Sloan, I.: High-dimensional integration: The quasi-monte
  carlo way.
\newblock Acta Numerica \textbf{22}, 133--288 (2013)

\bibitem{DICK-PILL-BOOK}
Dick, J., Pillichshammer, F.: Digital Nets and Sequences: Discrepancy Theory
  and Quasi-Monte Carlo Integration.
\newblock Cambridge University Press, Cambridge (2010)

\bibitem{DIFF-SET-FAN}
Fan, C.T., Siu, M.K., Ma, S.L.: Difference sets in dihedral groups and
  interlocking difference sets.
\newblock Ars Combin. \textbf{20.A}, 99--107 (1985)

\bibitem{GAP}
GAP: {NTL: GAP -Groups, Algorithms, Programming- a System for Computational
  Discrete Algebra}.
\newblock \url{https://www.gap-system.org/}

\bibitem{HANAKI}
Hanaki, A., Miyamoto, I.: Classification of association schemes with 16 and 17
  vertices.
\newblock Kyushu Journal of Mathematics \textbf{52}, 383 -- 395 (1998).
\newblock Electric data available from
  \url{http://math.shinshu-u.ac.jp/\~{}hanaki/as/}

\bibitem{KIBLER}
Kibler, R.: A summary of noncyclic difference sets, $k < 20$.
\newblock Journal of combinatorial theory, Series A \textbf{25}, 62--67 (1978)

\bibitem{MARTIN-STINSON}
Martin, W., Stinson, D.: Association schemes for ordered orthogonal arrays and
  (t, m, s)-nets.
\newblock Canadian Journal of Mathematics \textbf{51}, 326--346 (1999)

\bibitem{niederreiter:book}
Niederreiter, H.: Random Number Generation and Quasi-Monte Carlo Methods.
\newblock CBMS-NSF, Philadelphia, Pennsylvania (1992)

\bibitem{NIEDERREITER-PIRSIC}
Niederreiter, H., Pirsic, G.: Duality for digital nets and its applications.
\newblock Acta Arith. \textbf{97}, 173--182 (2001)

\bibitem{serre}
Serre, J.P.: Repr\'esentations line\'aires des groupes finis, 3rd edn.
\newblock Hermann, Paris (1978)

\bibitem{SLOAN-JOE}
Sloan, I., Joe, S.: Lattice methods for multiple integration.
\newblock Oxford University Press, New York (1994)

\bibitem{GODA-SUZUKI-QMC-survey}
Suzuki, K., Goda, T.: The state of the art in quasi-{Monte} {Carlo} methods.
\newblock Preprint, In Japanese

\end{thebibliography}


\end{document}